\documentclass[10pt]{amsart}
\usepackage{amssymb,amsthm,amsmath}
\usepackage{ifthen}
\usepackage{graphicx}
\usepackage{hyperref}
\nonstopmode
 \numberwithin{equation}{section}
\newtheorem{thm}{Theorem}[section]
\newtheorem{cor}[thm]{Corollary}
\newtheorem{lem}[thm]{Lemma}

\newtheorem{defn}[thm]{Definition}

\allowdisplaybreaks[2]

\theoremstyle{definition}
\newtheorem{rmk}[thm]{Remark}

\newenvironment{pf}[1][]{%
	\vskip 3mm
	\noindent
	\ifthenelse{\equal{#1}{}}%
	{{\slshape Proof. }}%
	{{\slshape #1.} }%
}%
{\qed\bigskip}

\newcounter{alphabet}
\newcounter{tmp}

\ifx\undefined\bysame
\newcommand{\bysame}{\leavevmode\hbox to3em{\hrulefill}\,}
\fi

\title[Schatten Class and nuclear pseudo-differential operators]
{Schatten Class and nuclear pseudo-differential operators on homogeneous spaces of compact groups}

\author{Vishvesh Kumar} 

\address{Vishvesh Kumar, Ph. D.  \endgraf Department of Mathematics: Analysis, Logic and Discrete Mathematics
	\endgraf Ghent University
	\endgraf Krijgslaan 281, Building S8,	B 9000 Ghent,
	Belgium .} 
\email{vishveshmishra@gmail.com}

\author{Shyam Swarup Mondal} 

\address{Shyam Swarup Mondal  \endgraf Department of Mathematics
	\endgraf IIT Guwahati
	\endgraf Guwahati, Assam, India.} 
\email{mondalshyam055@gmail.com}

\date{\today}
\begin{document}
	\allowdisplaybreaks
	
	\begin{abstract} 
		Given a compact (Hausdorff) group $G$ and a closed subgroup $H$ of $G,$ in this paper we present symbolic criteria for  pseudo-differential operators on compact homogeneous space $G/H$ characterizing the Schatten-von Neumann classes $S_r(L^2(G/H))$ for all $0<r \leq \infty.$ We go on to provide a symbolic characterization for  $r$-nuclear, $0< r \leq 1,$\, pseudo-differential operators on $L^{p}(G/H)$-space with applications to adjoint, product and trace formulae.  The criteria here are given in terms of the concept of matrix-valued symbols defined on noncommutative analogue of phase space $G/H \times \widehat{G/H}.$  Finally, we present applications of aforementioned results in the context of heat kernels.
	\end{abstract}

	\maketitle
	\tableofcontents 
	\section{Introduction}
	The theory of pseudo-differential operators is one of the most significant tools in modern mathematics to study the problems of geometry and of partial differential operators \cite{Hor}.
	The study of pseudo-differential operators is first introduced by  Kohn and Nirenberg \cite{Niren}. Ruzhansky and Turunen \cite{RuzT, RT13} studied pseudo-differential operators with matrix-valued symbols on compact (Hausdorff) groups. They introduced symbol classes and symbolic calculus for matrix-valued symbols on compact Lie groups and presented its applications in other directions also. Later, the theory of Pseudo-differential operators with matrix-valued symbols on compact (Hausdoff) groups, compact homogeneous spaces, compact manifolds is broadly studied by several authors \cite{  Cardona, CardonaKumar1,  CardonaKumar,   DK, Aparajita, D2, DR,  DRTk, Ruz, Ghaemi, Vish, KW, Shahla1, Shahla3, RuzT,  Wong2} in many context. 
	
	Let $G$ be a compact (Hausdorff) group and let $H$ be a closed subgroup of $G.$ In this paper we will mainly address the following problems: (1) To find criteria for pseudo-differential operator to be in  Schatten-von Neumann class $S_r$ of operators  on $L^2(G/H)$ for $0<r \leq \infty;$ (2) to find criteria for pseudo-differential operators from $L^{p_1}(G/H)$ into $L^{p_2}(G/H)$   to be $r$-nuclear, $0<r \leq 1,$  for $1 \leq p_1, \,p_2< \infty;$ and (3) applications to find a trace formula and to provide criteria for heat kernel  to be nuclear on $L^{p}(G/H).$ In order to do this, we will use the global quantization developed for compact homogeneous spaces as a non-commutative analogue of the Kohn-Nirenberg quantization of operators on $\mathbb{R}^n.$
	
	Recently, several researchers started a extensive research for finding the criteria for Schatten classes and $r$-nuclear operators in terms of symbols with lower regularity \cite{BT10, DR, DR18, Toft1,  Toft}. Ruzhansky and Delgado \cite{DR, Ruz}  successfully drop the regularity condition at least in their setting using the matrix-valued symbols instead of standard Kohn-Nirenberg quantization. Inspired by the work of Delgado and Ruzhansky, we present symbolic criteria for pseudo-differential operators on $G/H$ to be Schatten class using the matrix-valued symbols defined on $G/H \times \widehat{G/H},$ a noncommutative analogue of phase space.  It is  well known that in the setting of Hilbert spaces the class of $r$-nuclear operators agrees with the Schatten-von Neumann ideal of order $r$ \cite{Oloff}.  In general, for trace class operators on Hilbert spaces, the trace of an operator given by integration of its integral kernel over the diagonal is equal to the sum of its eigenvalues. However, this property fails in Banach spaces. The importance of $r$-nuclear operator lies in the work of Grothendieck, who proved that for  $2/3$-nuclear operators, the trace in Banach spaces agrees with the sum of all the eigenvalues with multiplicities counted. Therefore, the notion of $r$-nuclear operators becomes useful. Now, the question of finding good criteria for ensuring the $r$-nuclearity of operators arises but this has to be formulated in terms different from those on Hilbert spaces and has to take into account the impossibility of certain kernel formulations in view of Carleman's example \cite{Carl} (also see \cite{DR}). In view of this, we will establish conditions imposed on symbols instead of kernels ensuring the $r$-nuclearity of the corresponding operators.  
	
	The initiative of finding the necessary and sufficient conditions for a pseudo-differential operator defined on a group to be $r$-nuclear has been started by Delgado and Wong \cite{Delgado}. The main ingredient of such a characterization is a theorem of Delgado \cite{Delgado}. A multilinear version of Delgado's theorem recently proved by  first author and D. Cardona to study the nuclearity of multilinear pseudo-differential operators on the lattice and the torus \cite{CardonaKumar1, CardonaKumar}. In a seminal paper of Delgado and Ruzhansky \cite{DR}, they studied the $L^p$-nuclearity and traces of pseudo-differential operators on compact Lie groups using the global symbolic calculus developed by Ruzhansky and Turunen \cite{RuzT}. Later, they with their collaborators  extends these results to compact homogeneous spaces and compact manifolds \cite{ D2,DR, DRTk, Ruz}.  On the other hand, Wong and his collaborators extended the version of \cite{Delgado} in the settings of abstract compact groups with differential structure \cite{ Ghasemi, Ghaemi}. Characterizations of nuclear operators in terms of decomposition of symbol through Fourier transform were investigated by Ghaemi, Jamalpour Birgani and Wong for $\mathbb{S}^1$ \cite{Ghasemi}. Later they generalized their results on  nuclearity  to the pseudo-differential operators for any arbitrary compact group \cite{Ghaemi}.  
	
	The homogeneous spaces of abstract compact groups play an important role in mathematical physics, geometric analysis, constructive approximation and coherent state transform, see \cite{ {Kir1}, {Kis4},  {Kis3}, {Kis2},  {Kis1}, {Niren}} and the references therein. The study of pseudo-differential operators on homogeneous spaces of compact groups was started by the first author \cite{Vish}. We use the operator-valued Fourier transform on homogeneous spaces of compact groups developed by Ghani Farashahi \cite{Fara2}. By using this Fourier transform,  we introduce a global pseudo-differential calculus for homogeneous spaces of compact groups and study the Schatten class  of operator on $L^2(G/H)$ and $r$-nuclear  operators on $L^p$-spaces on compact homogeneous spaces. Our results can be seen as a compliment and  generalization  of the work of Delgado and Ruzhansky, pseudo-differential operators for compact Lie groups \cite{DR} as well as generalization of work of Ghaemi and Wong \cite{Ghasemi} on pseudo-differential operators on compact groups.        
	

	We begin this paper by recapitulating some basic  Fourier analysis on homogeneous spaces on compact groups from \cite{Fara2} in Section 2 although a parallel theory of homogeneous space of compact Lie groups can be found in classical book of Vilenkin \cite{Vilenkin} and recent papers and books \cite{Applebaum, AR, DR18, NRT16}. Later in this section we present a global quantization (Ruzhansky and Turunen \cite{RuzT}) on homogeneous spaces of compact groups related to a matrix-valued symbols.  In Section 3, we give  symbolic criteria of $r$-Schatten class operators defined on $L^2(G/H)$. In Section 4, we start our investigation on $r$-nuclear operator. We begin this section by providing sufficient conditions for an operator be $r$-nuclear in term of conditions on symbol of the operator. We also present a  characterization of $r$-nuclear pseudo-differential operator on  $L^p$-space for homogeneous spaces of compact groups. We calculate the nuclear trace of related pseudo-differential operators. In Section 5, we find a symbol for the adjoint of $r$-nuclear pseudo-differential operators on homogeneous space of compact groups and give a characterization for self-adjointness. We also compute the symbol of the product of a nuclear operator and a bounded linear operator. We end this paper by presenting applications of our results in the context of  heat kernels. 
	
	\section{Fourier analysis and the global quantization on homogeneous spaces of compact groups}
	
	We begin this section by recalling some basic and important facts of harmonic analysis on homogeneous spaces of compact (Hausdorff) groups from \cite{Fara2} which is almost similar to the theory given in \cite{Applebaum} and   \cite{Vilenkin} (see also \cite{ AR, DR18, NRT16}) for homogeneous spaces of compact Lie groups. 
	
	Let $G$ be a compact (Hausdorff) group with normalized Haar measure $dx$ and let $H$ be a closed subgroup of $G$ with probability Haar measure $dh.$ The left coset space $G/H$ can be seen as a homogeneous space with respect to the action of $G$ on $G/H$ given by left multiplication. Let $\mathcal{C}(\Omega)$ denote the space of continuous functions on a compact Hausdorff space $\Omega.$ Define $T_H: \mathcal{C}(G) \rightarrow \mathcal{C}(G/H)$ by $$T_H(f)(xH)= \int_H f(xh)\,dh, \,\,\,\, xH \in G/H.$$ Then $T_H$ is an onto map. 
	The homogeneous space $G/H$ has a unique normalized $G$-invariant positive Radon measure $\mu$ such that Weil formula $$ \int_{G/H}T_H(f)(xH)\, d\mu(xH)= \int_G f(x)\,dx$$ holds. The map $T_H$ can be extended to $L^2(G/H, \mu)$ and is a partial isometry on $L^2(G/H)$ with  $\left\langle T_{H}(f), T_{H}(g)\right\rangle_{L^{2}(G/H, \mu)}=\langle f, g\rangle_{L^{2}(G)}$ for all $f, g \in L^{2}(G) $.
	
	

	Let $(\pi, \mathcal{H}_\pi)$ be continuous unitary representation of compact group $G$ on a Hilbert space $\mathcal{H}_\pi.$ It is well-known that any irreducible representation $(\pi, \mathcal{H}_\pi)$ is finite dimensional with the dimension $d_\pi$ (say).  Consider the operator valued integral
	$$T_H^\pi := \int_H \pi(h)\, dh$$
	{defined in the weak sense, i.e., $\langle T_H^\pi u, v \rangle = \int_H \langle \pi(h)u, v\rangle\, dh,$ for all $u,v \in \mathcal{H}_\pi.$ Note that  $T_H^\pi$ is a bounded linear operator on $\mathcal{H}_\pi$ with norm bounded by one. Further, $T_H^\pi$ is a partial isometric orthogonal projection and $T_H^\pi$ is an identity operator  if and only if $\pi(h)=I$ for all $h \in H.$  
		\begin{defn} Let $H$ be a closed subgroup of a compact group $G.$ Then the dual object $\widehat{G/H}$ of $G/H$ is a subset of $\widehat{G}$ and  given by $$ \widehat{G/H}:= \left\lbrace \pi \in \widehat{G}: T_H^\pi \neq 0  \right\rbrace= \left\lbrace \pi \in \widehat{G}: \int_H \pi(h)\,dh \neq 0 \right\rbrace.$$ \end{defn}
		We would like to note here that the set $\widehat{G/H}$ is the set of all type $1$ representations of $G$ with respect to $H$ which was denoted by $\widehat{G_0}$ in \cite{ NRT16, Vilenkin}.
		
		Let $\pi \in \widehat{G/H}.$ Then the functions \(\pi_{\zeta, \xi}^{H} : G / H \rightarrow \mathbb{C}\) defined by
		$$\pi_{\zeta, \xi}^{H}(x H) :=\left\langle\pi(x) T_{H}^{\pi} \zeta, \xi\right\rangle \quad \text{for \(x H \in G / H, \)}$$
		for \(\xi, \zeta \in \mathcal{H}_{\pi},\) are called \(H\) -matrix elements of \(\left(\pi, \mathcal{H}_{\pi}\right)\). If $\{e_1, e_2, \dots, e_{d_\pi}\}$ is an orthonormal basis for $\mathcal{H}_{\pi}$ then we denote $\langle\pi(x) T_{H}^{\pi}e_i, e_j \rangle$ by $\pi_{ij}^H(xH).$ Now, using the orthogonality relation of matrix coefficients of $G$ and the fact   \(T_{H}\left(\pi_{\zeta, \xi}\right)=\pi_{\zeta, \xi}^{H}\)	 we have 
		$$  \left\langle \pi_{i,  j}^{H}, \xi_{k, l}^{H} \right\rangle_{L^{2}(G/H, \mu)}=\frac{1}{d_{\pi}}\delta_{\pi \xi} \delta_{ik}\delta_{jl}.$$

		Let \(\varphi \in L^{1}(G / H, \mu)\) and \(\pi \in \widehat{G / H} .\) Then the  group Fourier transform of \(\varphi\) at \(\pi\) is a bounded linear operator  defined by  
		\begin{align}\label{ft}
		\mathcal{F}_{G / H}(\varphi)(\pi)=\hat{\varphi}(\pi) :=\int_{G / H} \varphi(x H) \Gamma_{\pi}(x H)^{*} \,d \mu(x H)
		\end{align}
		on the Hilbert space \(\mathcal{H}_{\pi},\) where for \(x H \in G / H\) the notation \(\Gamma_{\pi}(x H)\) stands for the bounded linear operator on \(\mathcal{H}_{\pi}\) satisfying
		$$
		\left\langle\zeta, \Gamma_{\pi}(x H) \xi\right\rangle=\left\langle\zeta, \pi(x) T_{H}^{\pi} \xi\right\rangle
		$$
		for all \(\zeta, \xi \in \mathcal{H}_{\pi}\). 
		Note that from the notation of  $\Gamma_{\pi}(xH)$, the  $H$-matrix coefficients $ \pi_{i,  j}^{H}(xH)$ are same as $\Gamma_{\pi}(xH)_{ij}$.		Moreover if \(f \in L^{2}(G / H)\) then $\hat{\varphi}(\pi)$ is a Hilbert-Schmidt operator on $ \mathcal{H}_{\pi}$ and satisfies the following  Plancherel formula as stated in next theorem. 
		\begin{thm}
			For  \(\varphi \in L^{2}(G / H, \mu)\) we have
			$$\sum_{[\pi] \in \widehat{G / H}} d_{\pi}\|\hat{\varphi}(\pi)\|_{S_2}^{2}=\|\varphi\|_{L^{2}(G / H, \mu)}^{2},$$
			where $\|.\|_{S_2}$ stands for the Hilbert-Schmidt norm on the space of all Hilbert-Schmidt operators on $ \mathcal{H}_{\pi}.$
		\end{thm}
		\begin{thm}
			For  $\varphi \in L^{2}(G / H, \mu)$ the following Fourier inversion formula holds
			\begin{align} \label{inverse}
			\varphi(x H)=\sum_{[\pi] \in \widehat{G / H}} d_{\pi} \operatorname{Tr}[\hat{\varphi}(\pi) \pi(x) T_{H}^{\pi}] \quad \text { for } \mu-\text {a.e. } x H \in G / H.
			\end{align}
		\end{thm}
		
		We would like to record the following lemma whose  proof is similar to \cite[Lemma 2.5]{DR} by using the fact that the operator $T_H^\pi$ is norm bounded by one. 
		\begin{lem} \label{vis2.4}
			Let $G/H$ be a compact homogeneous space with normalized measure $\mu$ and let $\pi \in \widehat{G/H}.$ Then for all $1 \leq i, \, j \leq d_\pi$ we have 
			$$\|\Gamma_\pi(\cdot)_{ij}\|_{L^q(G/H)} \leq \begin{cases} d_{\pi}^{-\frac{1}{q}} & 2 \leq q \leq \infty, \\ d_\pi^{-\frac{1}{2}} & 1 \leq q \leq 2, \end{cases}$$ with the convention that for $q=\infty$ we have $d_\pi^{-\frac{1}{q}}=1.$
		\end{lem}
		
		Given a continuous  linear  operator  $T: C(G/H) \rightarrow C(G/H)$, its {\it  matrix-valued  global symbol} $ \sigma_T(xH, \pi)\in \mathbb{C}^{d_{\pi}\times  {d_{\pi}}}$ is defined by
		\begin{align}\label{600}
		T_{H}^{\pi}  \sigma_T(xH, \pi) =\pi(x)^* (T\Gamma_{\pi})(xH),
		\end{align}
		where $T\Gamma_{\pi}$ stands for the action of  $T$ on the  matrix components of $\Gamma_{\pi}(xH)$. 
		
		Thus setting
		$(T\Gamma_{\pi}(xH))_{mn}=(T (\Gamma_{{\pi}_{mn}} ) )(xH)$, we have
		\begin{align*}
		(T_{H}^{\pi}  \sigma_T(xH, \pi))_{mn} :=\sum_{k=1}^{d_\pi} \overline{\pi_{km}(x)} (T\Gamma_{\pi}(xH))_{kn},
		\end{align*}
		where $1\leq m, n\leq d_{\pi}$.

		Now, let the symbol $\sigma_T$ is a matrix-valued global symbol for continuous linear operator $T: C(G/H) \rightarrow C(G/H)$ as above. Then we can recover the operator $T$  by using the Fourier   inversion formula as follows:
		\begin{align*}
		Tf(x H) &=T \left( \sum_{[\pi] \in \widehat{G / H}} d_{\pi} \operatorname{Tr}( \pi(x) T_{H}^{\pi}\hat{f}(\pi)) \right)\\
		&= \sum_{[\pi] \in \widehat{G / H}} d_{\pi} \operatorname{Tr}( T \Gamma_\pi(xH)\,\hat{f}(\pi)).
		\end{align*}
		By using \eqref{600} and the relation $\pi(x)T_H^\pi =\Gamma_\pi(xH),$ we get 
		\begin{align} \label{vishquant}
		Tf(xH)=   \sum_{[\pi] \in \widehat{G / H}} d_{\pi} \operatorname{Tr}( \Gamma_\pi(xH) \sigma_T(xH, \pi) \hat{f}(\pi))
		\end{align}
		for all $f \in C(G/H),$  $\mu$-a.e. $xH \in G/H$ and the sum is independent of the representation from each equivalence class $[\pi] \in \widehat{G/H}.$ We will also write $T= \text{Op}(\sigma_T)$ for operator $T$ given by the formula \eqref{vishquant} and will be called a {\it pseudo-differential operator} corresponding to matrix-valued symbol $\sigma_T.$  For more details and consistent development of this quantization on compact Lie group and the corresponding  symbolic calculus we refer \cite{RuzT} and \cite{RT13}. 
		
		\begin{rmk}
			Let $H$ be a closed normal subgroup of the  compact group $G$ and let $\mu$ be the normalized $G$-invariant measure over the left quotient space $G/H$ associated to the weil's formula. Then $\mu$ is a Haar measure over the compact quotient group $G/H$ and  $\widehat{G/H}=H^{\perp}:=\{ \pi \in \widehat{G}: \pi(h)=I ~\text{for all } ~h\in H\}$.
			Moreover the Fourier transform (\ref{ft}), inverse Fourier transform (\ref{inverse}) and the pseudo-differential operator given in (\ref{vishquant}) coincide with the classical Fourier transform, Inverse Fourier transform and pseudo-differential operator over the compact quotient group $G/H$ respectively. 
		\end{rmk}
		
		\section{$r$-Schatten-von Neumann class of  pseudo-differential operators on $L^2(G/H)$} 
		
		This section is devoted to the study of $r$-Schatten-von Neumann  class of pseudo-differential operators on the Hilbert space $L^2(G/H).$  We begin this section with the definition of $r$-Schatten-von Neumann class of operators.
		
		If \(\mathcal{H}\) is a complex Hilbert space, a linear compact operator \(A : \mathcal{H} \rightarrow \mathcal{H}\) belongs to the $r$-Schatten-von Neumann class \(S_{r}(\mathcal{H})\) if
		$$
		\sum_{n=1}^{\infty}\left(s_{n}(A)\right)^{r}<\infty,
		$$where \(s_{n}(A)\) denote the singular values of \(A,\) i.e. the eigenvalues of \(|A|=\sqrt{A^{*} A}\)
		with multiplicities counted.
		
		For \(1 \leq r<\infty\),  the class \(S_{r}(\mathcal{H})\) is  a Banach space
		endowed with the norm
		$$
		\|A\|_{S_{r}}=\left(\sum_{n=1}^{\infty}\left(s_{n}(A)\right)^{r}\right)^{\frac{1}{r}}.
		$$
		
		For  \(0<r<1\), the $\|\cdot\|_{S_r}$ as above only defines a quasi-norm with respect to which
		\(S_{r}(\mathcal{H})\) is complete. An operator belongs to the class \(S_{1}(\mathcal{H})\) is known as {\it Trace class} operator. Also, an operator belongs to   \(S_{2}(\mathcal{H})\) is known as  {\it Hilbert-Schmidt} operator.
		
		
		Let $L^2(G/H \times  \widehat{G / H})$ denotes the space of all matrix-valued function $\sigma_A$ on $G/H \times  \widehat{G / H}$ such that  
		$$\|\sigma_A\|_{L^2(G/H \times  \widehat{G / H})}=\left(  \int_{G/H}  \sum_{[\xi] \in \widehat{G/H}}d_{\xi}  \| \sigma_A(xH, \xi) T_{H}^{\xi} \|_{S_{2}}^2  \,d\mu{(xH)}\right)^{\frac{1}{2}}<\infty.$$
		The following theorem gives a characterization of Hilbert-Schmidt pseudo-differential operators on $G/H$. We would like to remark here that the following theorem is already proved by the first author in \cite{Vish} using a different method.  
		\begin{thm}\label{110}
			Let $T : L^2 (G/H ) \rightarrow L^2 (G/H )$ be a continuous linear operator with the matrix-valued symbol \(\sigma_T\) on \(G/H \times \widehat{G/H}\). Then the operator $T$ is a Hilbert-Schmidt  operator if and only if 
			$	\sigma_T\in L^2(G/H \times  \widehat{G / H})$.  Moreover, we have  $$ \| T \|_{S_2} =   \|\sigma_T\|_{L^2(G/H \times  \widehat{G / H})}. $$
		\end{thm}
		\begin{pf}
			For all  $f \in   L^2(G/H)$, we have
			\begin{align*}
			Tf(xH)&= \sum_{[\xi] \in \widehat{G/H}} d_{\xi} \operatorname{Tr}(\Gamma_{\xi}(x H) \sigma_T(xH, \xi) \hat{f}(\xi))\\
			&=\int_{G/H}  \sum_{[\xi] \in \widehat{G/H}} \,d_{\xi} \operatorname{Tr}(\Gamma_{\xi}(x H) \sigma_T(xH, \xi)  \Gamma_{\xi}(w H)^* ) {f}(wH) \,d\mu{(wH)}\\
			&=\int_{G/H} K(xH, wH) {f}(wH) \,d\mu{(wH)},
			\end{align*} where the kernel is given by $$ K(xH, wH)=\displaystyle \sum_{[\xi] \in \widehat{G/H}} d_{\xi} \operatorname{Tr}(\Gamma_{\xi}(x H) \sigma_T(xH, \xi)  \Gamma_{\xi}(w H)^* ), ~~xH, wH \in G/H.$$

			We have \begin{align*}
			\|T\|_{\mathrm{S_2}}^{2}&=\int_{G/H} \int_{G/H}|K(xH, yH)|^{2} \, d\mu( xH) \,d\mu( yH)\\
			&=\int_{G/H} \int_{G/H}|K(xH, xz^{-1}H)|^{2} \,d\mu( xH) \,d\mu( zH).
			\end{align*}
			Now
			\begin{align*}
			K(xH, xz^{-1}H) &=\displaystyle \sum_{[\xi] \in \widehat{G/H}} d_{\xi} \operatorname{Tr}\left (\Gamma_{\xi}(x H) \sigma_T(xH, \xi)  \Gamma_{\xi}(xz^{-1}H)^* \right )\\
			&=\displaystyle \sum_{[\xi] \in \widehat{G/H}} d_{\xi} \operatorname{Tr}\left (\Gamma_{\xi}(zH) \sigma_T(xH, \xi)  T_{H}^{\xi}\right )\\
			&= \mathcal{F}^{-1}\tau(xH, \cdot) ( zH),
			\end{align*} where $\tau(xH, \xi)=\sigma_T(xH, \xi)  T_{H}^{\xi}$.
			Therefore, using Plancherel's formula, we have
			\begin{align*}
			\|T\|_{\mathrm{S_2}}^{2}&=\int_{G/H} \int_{G/H}|K(xH, xz^{-1}H)|^{2} d\mu( xH) \,d\mu( zH)\\
			&=\int_{G/H} \int_{G/H}|\mathcal{F}^{-1}\tau(xH, \cdot) ( zH)|^{2} d\mu( xH)\, d\mu( zH)\\
			&=\int_{G/H}  \sum_{[\xi] \in \widehat{G/H}}   d_{\xi}   \|\tau(xH, \xi)\|_{S_{2}}^2   \, d\mu( xH)\\
			&=\int_{G/H}  \sum_{[\xi] \in \widehat{G/H}}   d_{\xi}   \| \sigma_T(xH, \xi)  T_{H}^{\xi} \|_{S_{2}}^2   \, d\mu( xH)\\
			&=\|\sigma_T\|_{L^2(G/H \times  \widehat{G/H})}.
			\end{align*}\end{pf}
		
		The following lemma is a consequence of the definition of Schatten classes (see \cite{Ruz}) which is needed to obtain our main result.
		\begin{lem}\label{100}  Let \(A : \mathcal{H} \rightarrow \mathcal{H}\) be a linear compact operator. Let \(0<r, t<\infty .\) Then
			\(A \in S_{r}(\mathcal{H})\) if and only if \(|A|^{\frac{r}{t}} \in S_{t}(\mathcal{H}) .\) Moreover, \(\|A\|_{S_{r}}^{r} =\||A|^{\frac{r}{t}} \|_{S_{t}}^{t}\).
		\end{lem}
		The corollary below is the main result of this section which present a characterization of a  pseudo-differential operator on $L^2(G/H)$ to be a Schatten class operator. The proof follows from Lemma \ref{100} with $t=2$ and Theorem \ref{110}.
		\begin{cor}
			Let $T : L^2 (G/H ) \rightarrow L^2 (G/H )$ be a continuous linear operator with the matrix-valued symbol \(\sigma_T\) on \(G/H \times \widehat{G/H}\).  Then $T\in S_{r}\left(L^{2}(G/H)\right)  $ if and only if 
			$$  \int_{G/H}  \sum_{[\xi] \in \widehat{G/H}}d_{\xi}  \|\sigma_{|T|^{\frac{r}{2}}}(xH, \xi) T_{H}^{\xi} \|_{S_{2}}^2  \, d\mu{(xH)}<\infty.$$
		\end{cor}

		\section{Characterizations and traces of $r$-nuclear, $0<r \leq 1,$ pseudo-differential operators on $L^p(G/H)$}
		
		This section is devoted to study of $r$-nuclear operators on Banach spaces $L^p(G/H).$ Here we present a symbolic  characterization of $r$-nuclear operators and give a formula for  the nuclear trace of such operator. We will begin this section by recalling the basic notions of nuclear operators on Banach spaces.
		
		Let $0<r \leq 1$ and   \(T\) be a bounded linear operator from a complex Banach space \(X\) into
		another complex Banach space  \(Y\) such that there exist sequences \(\left\{x_{n}^{\prime}\right\}_{n=1}^{\infty}\) in
		the dual space \(X^{\prime}\) of \(X\) and \(\left\{y_{n}\right\}_{n=1}^{\infty}\) in \(Y\) such that
		$
		\sum_{n=1}^{\infty}\left\|x_{n}^{\prime}\right\|_{X^{\prime}}^r\left\|y_{n}\right\|_{Y}^r<\infty
		$
		and $$
		T x=\sum_{n=1}^{\infty} x_{n}^{\prime}(x) y_{n}, \quad x \in X.
		$$
		Then we call \(T : X \rightarrow Y\) a $r$-nuclear operator and if \(X=Y,\) then its nuclear trace  \(\operatorname{Tr}(T)\)
		is given by
		$$
		\operatorname{Tr}(T)=\sum_{n=1}^{\infty} x_{n}^{\prime}\left(y_{n}\right).
		$$
		It can be proved that the definition of a nuclear operator and the definition of
		the trace of a $r$-nuclear operator are independent of the choices of the sequences
		\(\left\{x_{n}^{\prime}\right\}_{n=1}^{\infty}\) and \(\left\{y_{n}\right\}_{n=1}^{\infty} .\) The following theorem is a characterization of $r$-nuclear operators on $\sigma$-finite measure spaces \cite{D2}.
		\begin{thm}\label{11}
			Let $0<r \leq 1$. Let \(\left(X_{1}, \mu_{1}\right)\) and \(\left(X_{2}, \mu_{2}\right)\) be \(\sigma\)-finite measure spaces. Then \(a\)
			bounded linear operator \(T : L^{p_{1}}\left(X_{1}, \mu_{1}\right) \rightarrow L^{p_{2}}\left(X_{2}, \mu_{2}\right), 1 \leq p_{1}, p_{2}<\infty,\) is
			$r$-nuclear if and only if there exist sequences \(\left\{g_{n}\right\}_{n=1}^{\infty}\) in \(L^{p_{1}'}\left(X_{1}, \mu_{1}\right)\) and \(\left\{h_{n}\right\}_{n=1}^{\infty}\)
			in \(L^{p_{2}}\left(X_{2}, \mu_{2}\right)\) such that for all \(f \in L^{p_{1}}\left(X_{1}, \,\mu_{1}\right)\)
			
			$$
			(T f)(x)=\int_{X_{1}} K(x, y) f(y) d \mu_{1}(y), \quad x \in X_{2},
			$$
			where
			$$
			K(x, y)=\sum_{n=1}^{\infty} h_{n}(x) g_{n}(y), \quad x \in X_{2}, y \in X_{1}
			$$ and
			$$
			\sum_{n=1}^{\infty}\left\|g_{n}\right\|_{L^{{p_{1}}^{\prime}}\left(X_{1}, \mu_{1}\right)}^r\left\|h_{n}\right\|_{L^{p_{2}}\left(X_{2}, \mu_{2}\right)}^r<\infty.
			$$
		\end{thm}

		Let $0<r \leq 1$. Let \((X, \mu)\) be a \(\sigma\)-finite measure space. Let \(T : L^{p}(X, \mu) \rightarrow L^{p}(X, \mu)\),
		\(1 \leq p<\infty,\) be a $r$-nuclear operator. Then by Theorem \ref{11}, we can find sequences
		\(\left\{g_{n}\right\}_{n=1}^{\infty}\) in \(L^{p^{\prime}}(X, \mu)\) and \(\left\{h_{n}\right\}_{n=1}^{\infty}\) in \(L^{p}(X, \mu)\) such that $$
		\sum_{n=1}^{\infty}\left\|g_{n}\right\|_{L^{p^{\prime}}(X, \mu)}^r\left\|h_{n}\right\|_{L^{p}(X, \mu)}^r<\infty
		$$ 
		and for all \(f \in L^{p}(X, \mu)\),\,\,
		$$\begin{aligned}(T f)(x) &=\int_{X} K(x, y) f(y) \,d \mu(y), \quad x \in X, \end{aligned}$$
		where
		$$
		K(x, y)=\sum_{n=1}^{\infty} h_{n}(x) g_{n}(y), \quad x, y \in X
		$$ and it satisfies 
		$$\int_{X} |K(x, y)|~d\mu(y)\leq \sum_{n=1}^{\infty } \left\|g_{n}\right\|_{L^{p^{\prime}}(X, \mu)}^r\left\|h_{n}\right\|_{L^{p}(X, \mu)}^r.$$
		
		The nuclear trace  $\operatorname{Tr}(T)$  of   $T : L^{p}(X, \mu)  \rightarrow L^{p}(X, \mu)$  is given by  
		\begin{align}\label{tr}
		\operatorname{Tr}(T) =\int_{X} K(x, x) \,d \mu(x). 
		\end{align}
		
		Now, we present a characterization of $r$-nuclear pseudo-differential operators from \(L^{p_{1}}(G/H)\) into \(L^{p_{2}}(G/H)\). 
		
		\begin{thm}
			Let $0<r \leq 1$ and let \(T : L^{p_{1}}(G/H) \rightarrow L^{p_{2}}(G/H), 1 \leq p_{1}, p_{2}<\infty,\)  be a continuous linear operator with the matrix-valued symbol \(\sigma_T\) on \(G/H \times \widehat{G/H}\). Suppose that the symbol  $\sigma_T$ satisfies 
			$$\sum_{[\xi] \in \widehat{G/H}} d_\xi^{2+\frac{r}{\tilde{p}_1}}\|\|\sigma_T(\cdot, \xi)^t  \|_{op(\ell^\infty, \, \ell^\infty)}\|^r_{L^{p_2}(G/H)}<\infty,$$ where $\tilde{p_1}=\text{min}\{2, p_1\}.$ Then the operator  $T$ is $r$-nuclear. 
		\end{thm}
		\begin{proof} Since the operator $T$ can be written as \begin{align*}
			Tf(xH)&= \int_{G/H}  \sum_{[\xi] \in \widehat{G/H}} \,d_{\xi} \operatorname{Tr}(\Gamma_{\xi}(x H) \sigma_T(xH, \xi)  \Gamma_{\xi}(w H)^* ) {f}(wH) \,d\mu{(wH)},
			\end{align*} the kernel of $T$ is given by $$ K(xH, wH)=\displaystyle \sum_{[\xi] \in \widehat{G/H}} d_{\xi} \operatorname{Tr}(\Gamma_{\xi}(x H) \sigma_T(xH, \xi)  \Gamma_{\xi}(w H)^* ).$$
			
			Now we write $$ \operatorname{Tr}(\Gamma_{\xi}(x H) \sigma_T (xH, \xi)  \Gamma_{\xi}(wH)^* ) = \sum_{i,j=1}^{d_{\xi}} (\Gamma_\xi(xH) \sigma_T(xH, \xi))_{ij}  \overline{\Gamma_{\xi}(wH)}_{ij}, $$ and set that 
			$h_{\xi, ij}(xH)= d_\xi (\Gamma_\xi(xH) \sigma_T(xH, \xi))_{ij} $ and $g_{\xi, ij}(wH)= (\Gamma_\xi(wH)^*)_{ji} = \overline{\Gamma_{\xi}(wH)}_{ij}.$
			
			We observe that 
			$$(\Gamma_\xi(xH) \sigma_T(xH, \xi))_{ij} = \sum_{k=1}^{d_\xi} \Gamma_\xi(xH)_{ik}\, \sigma_T(xH, \xi)_{kj}= \sum_{k=1}^{d_\xi} (\sigma_T(xH, \xi))^t_{jk} \, \Gamma_\xi(xH)_{ik} .$$
			
			By taking into account that $|\Gamma_{\xi}(xH)_{ik}| \leq 1,$ we get 
			\begin{align*}
			|(\Gamma_\xi(xH) \sigma_T(xH, \xi))_{ij}| &= \left|\sum_{k=1}^{d_\xi} (\sigma_T(xH, \xi))^t_{jk} \, \Gamma_\xi(xH)_{ik}   \right| \\ &\leq \|\sigma_T(xH, \xi)^t\|_{op(\ell^\infty, \ell^\infty)} \|(\Gamma_\xi(xH)_{i1}, \Gamma_\xi(xH)_{i2}, \ldots, \Gamma_\xi(xH)_{id_{\xi}})\|_{\ell^\infty} \\& \leq  \|\sigma_T(xH, \xi)^t\|_{op(\ell^\infty, \ell^\infty)}.
			\end{align*}
			Therefore we have 
			\begin{align*}
			\|h_{\xi, ij}(\cdot)\|_{L^{p_2}(G/H)}^r &= \| d_\xi \,(\Gamma_\xi(\cdot) \sigma_T(\cdot, \xi))_{ij}\|_{L^{p_2}(G/H)}^r \\ & \leq d_{\xi}^r \|   \|\sigma_T(\cdot, \xi)^t\|_{op(\ell^\infty, \ell^\infty)}\|_{L^{p_2}(G/H)}.
			\end{align*}
			If $p_1'$ denotes the Lebesgue conjugate of $p_1$ then we have $\frac{1}{\tilde{p_1}}+\frac{1}{\tilde{q_1}}=1$ where $\tilde{q_1}=\text{max}\{2, p_1'\}.$ By Lemma \ref{vis2.4} we have $\|\Gamma_\xi(\cdot)\|_{L^{p_1'}(G/H)}^r \leq d_\xi^{-\frac{r}{\tilde{p}_1'}}.$
			Therefore, 
			\begin{align*}
			\sum_{[\xi], i, j} \|g_{\xi, ij}(\cdot)\|_{L^{p_1'}(G/H)}^r \|h_{\xi, ij}(\cdot)\|_{L^{p_2}(G/H)}^r &\leq \sum_{[\xi]} d_\xi^{-\frac{r}{\tilde{p}_1'}} d_{\xi}^r  d_\xi^2\|   \|\sigma_T(\cdot, \xi)^t\|_{op(\ell^\infty, \ell^\infty)}\|_{L^{p_2}(G/H)} \\ & \leq \sum_{[\xi]} d_\xi^{2+\frac{r}{\tilde{p}_1}}   \|   \|\sigma_T(\cdot, \xi)^t\|_{op(\ell^\infty, \ell^\infty)}\|_{L^{p_2}(G/H)} <\infty.
			\end{align*}
			Hence, by invoking  Theorem \ref{11} it follows that $T$ is $r$-nuclear.  
		\end{proof}  
		
		Next theorem gives the necessary and sufficient conditions for an operator to be $r$-nuclear in  terms of symbol decomposition. 		
		
		\begin{thm}\label{55}
			Let $0<r \leq 1$ and let \(T : L^{p_{1}}(G/H) \rightarrow L^{p_{2}}(G/H), 1 \leq p_{1}, p_{2}<\infty,\)  be a continuous linear operator with the matrix-valued symbol \(\sigma_T\) on \(G/H \times \widehat{G/H}\). Then $T$ is $r$-nuclear if and only if there exist sequences \(\left\{g_{k}\right\}_{k=1}^{\infty} \in L^{p_{1}^{\prime}}(G/H)\)
			and \(\left\{h_{k}\right\}_{k=1}^{\infty} \in L^{p_{2}}(G/H)\) such that
			$$
			\sum_{k=1}^{\infty}\left\|g_{k}\right\|_{ L^{p_{1}^{\prime}}(G/H)}^r\left\|h_{k}\right\|_{ L^{p_{2}}(G/H)}^r<\infty
			$$ and 
			$$T_{H}^{\xi} \sigma_T(xH, \xi) =\xi(x)^* \sum_{k=1}^{\infty} h_{k}(xH) { \widehat{\overline{g_k}}(\xi) ^*}, ~~ (xH, \xi) \in G / H \times \widehat{G / H}.
			$$
		\end{thm}
		\begin{pf}
			Suppose that \(T : L^{p_{1}}(G/H) \rightarrow L^{p_{2}}(G/H)\) is $r$-nuclear, where \(1 \leq p_{1}, p_{2}<\infty .\) Then by Theorem \ref{11}, there exist sequences \(\left\{g_{k}\right\}_{k=1}^{\infty}\) in \(L^{p_{1}^{\prime}}(G/H)\) and \(\left\{h_{k}\right\}_{k=1}^{\infty}\) in \(L^{p_{2}}(G/H)\) such that
			$$
			\sum_{k=1}^{\infty}\left\|g_{k}\right\|_{ L^{p_{1}^{\prime}}(G/H)}^r\left\|h_{k}\right\|_{ L^{p_{2}}(G/H)}^r<\infty
			$$
			and for all $f \in L^{p_{1}}(G/H) $ we have 
			\begin{align}\label{1}\nonumber
			\left(T f\right)(x H)
			&=\sum_{[\pi] \in \widehat{G/H}} d_{\pi} \operatorname{Tr}(\Gamma_{\pi}(x H) \sigma_T(xH, \pi) \hat{f}(\pi))\\\nonumber
			&=\sum_{[\pi] \in \widehat{G/H}} d_{\pi} \sum_{i,j=1}^{d_{\pi}}(\Gamma_{\pi}(x H) \sigma_T(xH, \pi))_{ij} \hat{f}(\pi)_{ji}\\\nonumber
			&=\int_{G/H}\sum_{[\pi] \in \widehat{G/H}} d_{\pi} \sum_{i,j=1}^{d_{\pi}}(\Gamma_{\pi}(x H) \sigma_T(xH, \pi))_{ij} \overline{\Gamma_{\pi}(w H)_{ij}} f(wH)~ d\mu(wH) \\
			&=\int_{G/H} \left(  \sum_{k=1}^{\infty} h_{k}(xH) g_{k}(wH)\right)   f(wH)~ d\mu(wH) 
			\end{align}
			for all \(xH \in G/H .\) Let \(\xi\) be a fixed but arbitrary element in \(\widehat{G/H} .\) Then for \(1 \leq\)
			\(m, n \leq d_{\xi} \) we define the function \(f\) on \(G/H\) by
			$$
			f(wH)=\Gamma_{\xi}(wH)_{nm}, \quad wH \in G/H.
			$$
			Since $$\int_{G/H} {\Gamma_{\xi}(wH)_{nm}}\overline{\Gamma_{\pi}(wH)_{ij}}  ~d\mu(wH)=\frac{1}{d_{\xi}}$$ if and only if $\pi=\xi$, $i=n$ and $j=m,$ and is zero otherwise, it follows from (\ref{1}) that	
			\begin{align*}
			(\Gamma_{\xi}(x H) \sigma_T(xH, \xi))_{nm}&= \sum_{k=1}^{\infty} h_{k}(xH) \left( \int_{G/H} g_{k}(wH)~ \Gamma_{\xi}(wH)_{nm} ~d\mu(wH)\right)\\
			&= \sum_{k=1}^{\infty} h_{k}(xH)  \overline{\left ( \widehat{\overline{g_k}}(\xi)\right )_{mn}}.
			\end{align*}
			Therefore,  \begin{align*}
			T_{H}^{\xi} \sigma_T(xH, \xi) =\xi(x)^*  \sum_{k=1}^{\infty} h_{k}(xH)  { \widehat{\overline{g_k}}(\xi) ^*}, ~~(xH, \xi) \in G / H \times \widehat{G / H}.
			\end{align*}
			
			Conversely, suppose that there exist sequences \(\left\{g_{k}\right\}_{k=1}^{\infty}\) in \(L^{p_{1}^{\prime}}(G/H)\) and \(\left\{h_{k}\right\}_{k=1}^{\infty}\)
			in \(L^{p_{2}}(G/H)\) such that
			$$
			\sum_{k=1}^{\infty}\left\|g_{k}\right\|_{ L^{p_{1}^{\prime}}(G/H)}^r\left\|h_{k}\right\|_{ L^{p_{2}}(G/H)}^r<\infty
			$$
			and
			$$
			T_{H}^{\xi} \sigma_T(xH, \xi) =\xi(x)^*  \sum_{k=1}^{\infty} h_{k}(xH) { \widehat{\overline{g_k}}(\xi) ^*}, ~~(xH, \xi) \in G / H \times \widehat{G / H}.
			$$
			Then, for all \(f \in L^{p_{1}}(G/H)\)
			\begin{align*}
			\left(T f\right)(x H)
			&=\sum_{[\xi] \in \widehat{G/H}} d_{\xi} \operatorname{Tr}(\Gamma_{\xi}(x H) \sigma_T(xH, \xi) \hat{f}(\xi))\\
			&=\sum_{[\xi] \in \widehat{G/H}} d_{\xi} \sum_{m,n=1}^{d_{\xi}}(\Gamma_{\xi}(x H) \sigma_T(xH, \xi))_{nm} \hat{f}(\xi)_{mn}\\
			&=\sum_{[\xi] \in \widehat{G/H}} d_{\xi} \sum_{m,n=1}^{d_{\xi}} \left( \sum_{k=1}^{\infty} h_{k}(xH)  { \widehat{\overline{g_k}}(\xi)^*}_{nm} \right ) \hat{f}(\xi)_{mn}\\
			&=\sum_{[\xi] \in \widehat{G/H}} d_{\xi} \sum_{m,n=1}^{d_{\xi}} \left( \sum_{k=1}^{\infty} h_{k}(xH)  \overline{ \widehat{\overline{g_k}}(\xi)_{mn} } \right ) \hat{f}(\xi)_{mn}\\
			&=\sum_{[\xi] \in \widehat{G/H}} d_{\xi} \sum_{m,n=1}^{d_{\xi}}  \sum_{k=1}^{\infty} h_{k}(xH) \left( \int_{G/H} g_{k}(wH)  \Gamma_{\xi}(w H)_{nm}~  d\mu(wH)  \right ) \hat{f}(\xi)_{mn}\\
			&=   \int_{G/H} \left( \sum_{[\xi] \in \widehat{G/H}} d_{\xi} \sum_{m,n=1}^{d_{\xi}}    \Gamma_{\xi}(w H)_{nm}  \hat{f}(\xi)_{mn}  \right) \sum_{k=1}^{\infty} h_{k}(xH)  g_{k}(wH) ~  d\mu(wH)  \\
			&=   \int_{G/H} \left( \sum_{[\xi] \in \widehat{G/H}} d_{\xi}   \operatorname{Tr}(  \Gamma_{\xi}(w H)  \hat{f}(\xi)) \right)     \sum_{k=1}^{\infty} h_{k}(xH)  g_{k}(wH) ~  d\mu(wH)  \\
			&=   \int_{G/H}\left (   \sum_{k=1}^{\infty} h_{k}(xH)  g_{k}(wH)\right ) f(wH)  ~  d\mu(wH) 
			\end{align*}
			for all \(xH \in G/H .\) Therefore by Theorem \ref{11}, it follows that $T$   is  $r$-nuclear.
		\end{pf}

		In the next theorem, we will give another characterization of $r$-nuclear pseudo-differential operators from  $L^{p_{1}}(G/H)$ into $L^{p_{2}}(G/H)$ in order to find  trace of $r$-nuclear operators from $L^{p}(G/H)$ into $L^{p}(G/H)$.
		
		
		\begin{thm}\label{22}
			Let $0<r \leq 1$ and let \(T : L^{p_{1}}(G/H) \rightarrow L^{p_{2}}(G/H), 1 \leq p_{1}, p_{2}<\infty,\)  be a continuous linear operator with the matrix-valued symbol \(\sigma_T\) on \(G/H \times \widehat{G/H}\). Then the pseudo-differential operator \(T \)
			is  $r$-nuclear   if and only if there exist sequences
			\(\left\{g_{k}\right\}_{k=1}^{\infty}\) in \(L^{p_{1}'}(G/H)\) and \(\left\{h_{k}\right\}_{k=1}^{\infty}\) in \(L^{p_{2}}(G/H)\) such that
			$$\sum_{k=1}^{\infty}\left\|g_{k}\right\|_{ L^{p_{1}^{\prime}}(G/H)}^r\left\|h_{k}\right\|_{ L^{p_{2}}(G/H)}^r<\infty$$
			and 
			$$
			\sum_{[\xi] \in \widehat{G/H}} d_{\xi} \operatorname{Tr}(\Gamma_{\xi}(x H) \sigma_T(xH, \xi) \Gamma_{\xi}(y H)^* )=\sum_{k=1}^{\infty} h_{k}(xH){g_k}(yH).
			$$
		\end{thm}
		
		\begin{pf}
			Suppose that \(T : L^{p_{1}}(G/H) \rightarrow L^{p_{2}}(G/H)\) is a $r$-nuclear operator for $1\leq p_1, p_2 <\infty$. Then by
			Theorem \ref{55}, there exist sequences \(\left\{g_{k}\right\}_{k=1}^{\infty}\) in \(L^{p_{1}^{\prime}}(G/H)\) and \(\left\{h_{k}\right\}_{k=1}^{\infty}\) in \(L^{p_{2}}(G/H)\)
			such that
			$$
			\sum_{k=1}^{\infty}\left\|g_{k}\right\|_{ L^{p_{1}^{\prime}}(G/H)}^r\left\|h_{k}\right\|_{ L^{p_{2}}(G/H)}^r<\infty
			$$
			and $$(\Gamma_{\xi}(x H) \sigma_T(xH, \xi))_{nm}= \sum_{k=1}^{\infty} h_{k}(xH)  \overline{\left ( \widehat{\overline{g_k}}(\xi)\right )}_{mn},~ (xH, \xi ) \in G / H \times \widehat{G / H},$$
			for all $n, m$ with   $1 \leq  n, m \leq d_{\xi}$. 
			
			Let  $yH\in  G/H$. Then
			\begin{align*}
			(\Gamma_{\xi}(x H) \sigma_T(xH, \xi))_{nm} \overline{\Gamma_{\xi}(y H)}_{nm} &= \sum_{k=1}^{\infty} h_{k}(xH)  \overline{\left ( \widehat{\overline{g_k}}(\xi)\right )}_{mn} \overline{\Gamma_{\xi}(y H)}_{nm} \\
			&= \int_{G/H} {\Gamma_{\xi}(z H)_{nm}} \overline{\Gamma_{\xi}(y H)}_{nm} \sum_{k=1}^{\infty} h_{k}(xH){g_k}(zH)  d\mu(zH).
			\end{align*}
			So we have  
			\begin{align*}
			&\sum_{m,n=1}^{d_{\xi}} (\Gamma_{\xi}(x H) \sigma_T(xH, \xi))_{nm} \overline{\Gamma_{\xi}(y H)}_{nm}\\
			&= \int_{G/H}\left(  \sum_{m,n=1}^{d_{\xi}} {\Gamma_{\xi}(z H)}_{nm} \overline{\Gamma_{\xi}(y H)}_{nm}\right)  \sum_{k=1}^{\infty} h_{k}(xH){g_k}(zH)  d\mu(zH).
			\end{align*}
			Therefore,  for all $xH,\, yH \in G/H$, we get
			\begin{align*}
			&\sum_{[\xi] \in \widehat{G/H}} d_{\xi} \operatorname{Tr}(\Gamma_{\xi}(x H) \sigma_T(xH, \xi){\Gamma_{\xi}(y H)^*})\\
			&= \int_{G/H} \left( \sum_{[\xi] \in \widehat{G/H}} d_{\xi} \operatorname{Tr}\left ( {\Gamma_{\xi}(z H)} {\Gamma_{\xi}(y H)^*}\right )\right )  \sum_{k=1}^{\infty} h_{k}(xH){g_k}(zH)  d\mu(zH) \\ &= \int_{G/H} \left( \sum_{[\xi] \in \widehat{G/H}} d_{\xi} \overline{\operatorname{Tr}\left ( {\Gamma_{\xi}(y H)} {\Gamma_{\xi}(z H)^*}\right )}\right )   \sum_{k=1}^{\infty} h_{k}(xH){g_k}(zH)  d\mu(zH) \\ &= \sum_{k=1}^{\infty} h_{k}(xH) \int_{G/H} \left( \sum_{[\xi] \in \widehat{G/H}} d_{\xi} \overline{\operatorname{Tr}\left ( {\Gamma_{\xi}(y H)} {\Gamma_{\xi}(z H)^*}\right )} {g_k}(zH)\right )\\ &= \sum_{k=1}^{\infty} h_{k}(xH) \int_{G/H} \left( \sum_{[\xi] \in \widehat{G/H}} d_{\xi} \overline{\operatorname{Tr}\left ( {\Gamma_{\xi}(y H)} {\Gamma_{\xi}(z H)^*} \overline{ {g_k}(zH)}\right )}\right ) \\ &= \sum_{k=1}^{\infty} h_{k}(xH) \overline{\left( \sum_{[\xi] \in \widehat{G/H}} d_{\xi} \operatorname{Tr}\left ( {\Gamma_{\xi}(y H)} \widehat{\overline{g}}(\xi) \right )\right )} = \sum_{k=1}^{\infty} h_{k}(xH) \overline{\left( \overline{g(yH)}\right )} \\ &= \sum_{k=1}^{\infty} h_{k}(xH) g_k(yH)
			\end{align*}
			for all $xH, yH$ in $G/H.$
			
			Conversely, let \(\left\{g_{k}\right\}_{k=1}^{\infty}\) and \(\left\{h_{k}\right\}_{k=1}^{\infty}\) be sequences in
			\(L^{p_{1}^{\prime}}(G/H)\) and \(L^{p_{2}}(G/H)\)  such that
			$$
			\sum_{k=1}^{\infty}\left\|g_{k}\right\|_{ L^{p_{1}^{\prime}}(G/H)}^r\left\|h_{k}\right\|_{ L^{p_{2}}(G/H)}^r<\infty
			$$
			and for all \(xH\) and \(yH\) in \(G/H,\)
			$$
			\sum_{[\xi] \in \widehat{G/H}} d_{\xi} \operatorname{Tr}(\Gamma_{\xi}(x H) \sigma(xH, \xi) \Gamma_{\xi}(y H)^* )=\sum_{k=1}^{\infty} h_{k}(xH){g_k}(yH).
			$$
			Then, for all $f\in  L^{p_1} (G/H)$
			\begin{align*}
			\left(T_{\sigma} f\right)(x H) &=\sum_{[\xi] \in \widehat{G/H}} d_{\xi} \operatorname{Tr}(\Gamma_{\xi}(x H) \sigma(xH, \xi) \hat{f}(\xi))\\
			&=\sum_{[\xi] \in \widehat{G/H}} d_{\xi} \sum_{m,n=1}^{d_{\xi}}(\Gamma_{\xi}(x H) \sigma(xH, \xi))_{mn} \hat{f}(\xi)_{nm}\\
			&=\int_{G/H} \left( \sum_{[\xi] \in \widehat{G/H}} d_{\xi}  \sum_{m,n=1}^{d_{\xi}}(\Gamma_{\xi}(x H) \sigma(xH, \xi))_{mn} \overline{\Gamma_{\xi}(yH)_{mn}} \right) f(yH)~d\mu(yH)\\
			&=\int_{G/H} \left( \sum_{[\xi] \in \widehat{G/H}} d_{\xi} \operatorname{Tr}(\Gamma_{\xi}(x H) \sigma(xH, \xi) \Gamma_{\xi}(yH)^* )\right ) f(yH)~d\mu(yH)\\
			&=\int_{G/H} \left(\sum_{k=1}^{\infty} h_{k}(xH){g_k}(yH)\right ) f(yH)~d\mu(yH)\\
			\end{align*}
			for all $xH\in G/H.$  This completes the proof.
		\end{pf}

		An immediate consequence of Theorem \ref{22} gives the trace of a $r$-nuclear pseudo-differential operator on $L^p(G/H)$ for $1 \leq  p < \infty$. Indeed, we have the following result.
		\begin{cor}\label{500}
			Let $0<r \leq 1$ and let \(T : L^{p}(G/H) \rightarrow L^{p}(G/H), 1 \leq p <\infty,\)  be a $r$-nuclear operator with the matrix-valued symbol \(\sigma_T\) on \(G/H \times \widehat{G/H}\).
			Then the nuclear trace  of \(T\) is given by
			$$
			\operatorname{Tr}\left(T\right)= \int_{G/H}\sum_{[\xi] \in \widehat{G/H}} d_{\xi} \operatorname{Tr}( T_{H}^{\xi} \sigma_T(xH, \xi)) d \mu(xH).
			$$
		\end{cor}
		\begin{pf}
			Using trace formula (\ref{tr}) and Theorem \ref{22} we have 
			\begin{align*} \operatorname{Tr}\left(T\right) 
			&=\int_{G/H} \sum_{k=1}^{\infty} h_{k}(xH) g_{k}(xH) d \mu(xH) \\ 
			&=\int_{G/H} \sum_{[\xi] \in \widehat{G/H}} d_{\xi} \operatorname{Tr}(\Gamma_{\xi}(x H) \sigma_T(xH, \xi) \Gamma_{\xi}(x H)^* ).
			\end{align*}
			Since $T_{H}^{\xi}$ is an orthogonal projection therefore,
			\begin{align*} \operatorname{Tr}\left(T\right) 
			&=\int_{G/H}  \sum_{[\xi] \in \widehat{G/H}} d_{\xi} \operatorname{Tr}( T_{H}^{\xi}\sigma_T(xH, \xi)) d \mu(xH).
			\end{align*} \end{pf}

		\section{Adjoint and product of nuclear pseudo-differential operators}
		In this section we give a  formula for the symbols of the adjoints of $r$-nuclear pseudo-differential operators from $L^{p_1}(G/H)$ into $L^{p_2}(G/H)$ for $1 \leq p_1, p_2 <\infty$,
		where $G$ is a compact Hausdorff group and $H$ be a closed subgroup of $G.$

		\begin{thm}\label{66}
			Let $0<r \leq 1$ and let \(T : L^{p_{1}}(G/H) \rightarrow L^{p_{2}}(G/H), 1 \leq p_{1}, p_{2}<\infty,\)  be a continuous linear operator with the matrix-valued symbol \(\sigma_T\) on \(G/H \times \widehat{G/H}\). Then the adjoint $T^*$ of $T$  is also a $r$-nuclear operator from $L^{p_{2}'}(G/H)$ into $ L^{p_{1}'}(G/H)$ with symbol $\tau$ given by 
			$$T_{H}^{\xi} \tau(xH, \xi) ={\xi}(x )^* \sum_{k=1}^{\infty} \overline{g_{k}}(xH){ \widehat{{h_k}}(\xi)^*}, ~~ (xH, \xi) \in G / H \times \widehat{G / H},
			$$ where \(\left\{g_{k}\right\}_{k=1}^{\infty}\) and \(\left\{h_{k}\right\}_{k=1}^{\infty}\)  are two sequences in \(L^{p_{1}^{\prime}}(G/H)\) and \(L^{p_{2}}(G/H)\) respectively  such that
			$$
			\sum_{k=1}^{\infty}\left\|h_{k}\right\|_{L^{p_{2}}(G/H)}^r\left\|g_{k}\right\|_{L^{{p_{1}}^{\prime}}(G/H)}^r<\infty.
			$$
		\end{thm}
		\begin{pf}
			For $f \in  L^{p_{1}}(G/H)$ and $g\in L^{p_{2}'}(G/H)$, from the definition of the adjoint of an operator we have 
			$$\int_{G/H}\left( T f \right)(x H)\overline{g(xH)}~ d\mu(xH)= \int_{G/H} f(x H) \overline{\left( T^*g\right) (xH)}~ d\mu(xH).$$
			Therefore,
			\begin{align}\label{33}\nonumber
			&\int_{G/H}\Biggl(   \int_{G/H}   \sum_{[\xi] \in \widehat{G/H} } d_{\xi} \sum_{m,n=1}^{d_{\xi}}(\Gamma_{\xi}(x H) \sigma_T(xH, \xi))_{mn} \\
			&\quad \quad \quad  \times   \overline{\Gamma_{\xi}(yH)_{mn}}  f(yH)~d\mu(yH)   \Biggr) \overline{g(xH)}~ d\mu(xH)\\\nonumber
			&= \int_{G/H} f(x H) \overline{\Biggl ( \int_{G/H} \sum_{[\xi] \in\widehat{G/H}} d_{\xi}  \sum_{m,n=1}^{d_{\xi}}(\Gamma_{\xi}(x H) \tau(xH, \xi))_{mn}}\\\nonumber
			&\quad \quad \quad  \times \overline{\overline{\Gamma_{\xi}(yH)_{mn}}  g(yH)~d\mu(yH) \Biggr)}~ d\mu(xH).
			\end{align}
			Now, let $\gamma $
			and $\eta$ be elements in $\widehat{G/H}$. Then for $1 \leq i, j \leq d_{\gamma}$ and $1 \leq p, q \leq d_{\eta}$, 
			we define the functions  $f$ and $g$ on $G/H$ by
			$$f(xH)=\Gamma_{\gamma}(xH)_{ij}, \quad xH \in G/H$$ and 
			$$g(xH)=\Gamma_{\eta}(xH)_{pq}, \quad xH \in G/H.$$

			Therefore, from (\ref{33}) it folllows that
			\begin{align*}
			&\int_{G/H}   (\Gamma_{\gamma}(x H) \sigma_T(xH, \gamma))_{ij} \overline{\Gamma_{\eta}(xH)}_{pq}~ d\mu(xH)\\
			&= \int_{G/H}\Gamma_{\gamma}(xH)_{ij} \overline{  (\Gamma_{\eta}(x H) \tau(xH, \eta))}_{pq} ~ d\mu(xH)
			\end{align*}
			and so 
			\begin{align*}
			&\overline{\int_{G/H}   (\Gamma_{\gamma}(x H) \sigma_T(xH, \gamma))_{ij} \overline{\Gamma_{\eta}(xH)_{pq}}~ d\mu(xH)}\\
			&= \int_{G/H} {(\Gamma_{\eta}(x H) \tau(xH, \eta))_{pq} }  \overline{\Gamma_{\gamma}(xH)_{ij} }  ~ d\mu(xH).
			\end{align*}
			Thus,
			\begin{align}\label{44}
			\overline{(((\Gamma_{\gamma}(\cdot ) \sigma_T(\cdot , \gamma))_{ij})^{\wedge}     ({\eta}))_{qp}}
			=(({(\Gamma_{\eta}(\cdot ) \tau(\cdot , \eta))_{pq} })^{\wedge} (\gamma))_{ji} 
			\end{align}
			for $1 \leq i, j \leq d_{\gamma}$,  $1 \leq p, q \leq d_{\eta}$ and all $\gamma $
			and $\eta$  in $\widehat{G/H}$. Since  \(T : L^{p_{1}}(G/H) \rightarrow L^{p_{2}}(G/H)\) is $r$-nuclear,  by Theorem \ref{55}, there exist sequences \(\left\{g_{k}\right\}_{k=1}^{\infty}\) in \(L^{p_{1}^{\prime}}(G/H)\) and \(\left\{h_{k}\right\}_{k=1}^{\infty}\) in \(L^{p_{2}}(G/H)\) such that
			$$
			\sum_{k=1}^{\infty}\left\|h_{k}\right\|_{L^{p_{2}}(G/H)}^r\left\|g_{k}\right\|_{L^{{p_{1}}^{\prime}}(G/H)}^r<\infty
			$$
			and for all $(yH, \gamma ) \in G / H \times \widehat{G / H}$, 	$$(\Gamma_{\gamma}(y H) \sigma_T(yH, \gamma))_{ij}= \sum_{k=1}^{\infty} h_{k}(yH)  \overline{\left ( \widehat{\overline{g_k}}(\gamma)\right )_{ji}}.$$
			So, for all $(xH, \eta ) \in G / H \times \widehat{G / H}$
			\begin{align*}
			(\Gamma_{\eta}(x H) \tau(xH, \eta))_{pq} &=\sum_{[\gamma] \in \widehat{G/H}} d_{\gamma} \operatorname{Tr}[\Gamma_{\gamma}(x H) ((\Gamma_{\eta}(\cdot ) \tau(\cdot , \eta))_{pq} )^{\wedge} (\gamma) ]\\&=\sum_{[\gamma] \in \widehat{G/H}}  \sum_{i,j=1}^{d_{\gamma}}d_{\gamma} (\Gamma_{\gamma}(x H))_{ij} (((\Gamma_{\eta}(\cdot ) \tau(\cdot , \eta))_{pq} )^{\wedge} (\gamma))_{ji}.
			\end{align*}
			Therefore, for all  $(xH, \eta ) \in G / H \times \widehat{G / H}$, we get from (\ref{44})

			\begin{align*}
			& (\Gamma_{\eta}(x H) \tau(xH, \eta))_{pq} \\
			&=\sum_{[\gamma] \in \widehat{G/H}}  \sum_{i,j=1}^{d_{\gamma}}d_{\gamma} (\Gamma_{\gamma}(x H))_{ij} \overline{(((\Gamma_{\gamma}(\cdot ) \sigma_T(\cdot , \gamma))_{ij})^{\wedge}     ({\eta}))_{qp}}\\
			&=\sum_{[\gamma] \in \widehat{G/H}}  \sum_{i,j=1}^{d_{\gamma}}d_{\gamma} (\Gamma_{\gamma}(x H))_{ij} \int_{G/H} \overline{(\Gamma_{\gamma}(yH ) \sigma_T(yH , \gamma))_{ij}}     {\Gamma_\eta(yH)}_{pq} ~d\mu(yH)\\
			&=\sum_{[\gamma] \in \widehat{G/H}}  \sum_{i,j=1}^{d_{\gamma}}d_{\gamma} (\Gamma_{\gamma}(x H))_{ij} \int_{G/H} \sum_{k=1}^{\infty} \overline{h_{k}(yH)}  { ( \widehat{\overline{g_k}}(\gamma))_{ji}}   {\Gamma_\eta(yH)}_{pq} ~d\mu(yH)\\
			&= \sum_{k=1}^{\infty} \left (  \int_{G/H} \overline{h_{k}(yH)}   \Gamma_\eta(yH)_{pq}~d\mu(yH)\right)  \sum_{[\gamma] \in \widehat{G/H}} \sum_{i,j=1}^{d_{\gamma}}d_{\gamma} (\Gamma_{\gamma}(x H))_{ij} { ( \widehat{\overline{g_k}}(\gamma))_{ji}}  \\
			&= \sum_{k=1}^{\infty}  \overline{\widehat{h_{k}}(\eta)_{qp}  }   \sum_{[\gamma] \in \widehat{G/H}} \sum_{i,j=1}^{d_{\gamma}}d_{\gamma} (\Gamma_{\gamma}(x H))_{ij} { ( \widehat{\overline{g_k}}(\gamma))_{ji}}  \\
			&= \sum_{k=1}^{\infty}  \overline{\widehat{h_{k}}(\eta)_{qp}  }   \sum_{[\gamma] \in \widehat{G/H}} d_{\gamma} \operatorname{Tr} (\Gamma_{\gamma}(x H) { \widehat{\overline{g_k}}(\gamma)})  \\
			&= \sum_{k=1}^{\infty}  \overline{\widehat{h_{k}}(\eta)_{qp}  }  \overline{g_k}(xH) =\sum_{k=1}^{\infty}  {\widehat{h_{k}}(\eta)_{pq}^*  }  \overline{g_k}(xH) 
			\end{align*}
			for all $1\leq p, q\leq d_{\eta}$. 
			
			Therefore, for all $(xH, \eta ) \in G / H \times \widehat{G / H}$, we get
			\begin{align*}
			\Gamma_{\eta}(x H) \tau(xH, \eta)= \sum_{k=1}^{\infty}  {\widehat{h_{k}}(\eta)^*} ~ \overline{g_k}(xH)  
			\end{align*}
			and hence 
			\begin{align*}
			T_{H}^{\eta} \tau(xH, \eta)={\eta}(x )^*  \sum_{k=1}^{\infty}  {\widehat{h_{k}}(\eta)^*} ~ \overline{g_k}(xH).
			\end{align*}
		\end{pf}
		
		As an application of Theorem \ref{55} and Theorem \ref{66}, in the next corollary we   give a  criterion for the self-adjointness of $r$-nuclear pseudo-differential operators.
		\begin{cor}\label{501}
			Let $0<r \leq 1$ and let  \(T : L^{2}(G/H) \rightarrow L^{2}(G/H)\)  be a   $r$-nuclear continuous linear operator with the matrix-valued symbol \(\sigma_T\) on \(G/H \times \widehat{G/H}\). 
			Then $T$ is self-adjoint if and only if there exist sequences 	\(\left\{g_{k}\right\}_{k=1}^{\infty}\)  and \(\left\{h_{k}\right\}_{k=1}^{\infty}\) in \(L^{2}(G/H)\) such that
			$$\sum_{k=1}^{\infty}\left\|h_{k}\right\|_{L^{2}(G/H)}^r\left\|g_{k}\right\|_{L^{2}(G/H)}^r<\infty,$$
			$$ \sum_{k=1}^{\infty} h_{k}(xH) { \widehat{\overline{g_k}}(\xi) ^*}= \sum_{k=1}^{\infty}  {\widehat{h_{k}}(\xi)^*} ~ \overline{g_k}(xH),~~ (xH, \xi) \in G / H \times \widehat{G / H},$$
			and $$T_{H}^{\xi}  \sigma_T(xH, \xi)= {\xi}(x)^* \sum_{k=1}^{\infty} h_{k}(xH) { \widehat{\overline{g_k}}(\xi) ^*}, ~~(xH, \xi) \in G / H \times \widehat{G / H}.$$	
		\end{cor}
		We can give another formula  for the adjoints  of $r$-nuclear operators in terms of symbols. Indeed, we have the following theorem.
		
		\begin{thm}
			Let $0<r \leq 1$. Let \(\sigma_T\) be a matrix-valued function on \(G/H \times \widehat{G/H}\) such that  the corresponding pseudo-differential operator \(T : L^{p_{1}}(G/H) \rightarrow L^{p_{2}}(G/H)\)
			is  $r$-nuclear for $1\leq p_{1}, p_{2}<\infty$.  Then the symbol   $\tau$  of the adjoint  $T^*: L^{p_{2}'}(G/H) \rightarrow L^{p_{1}'}(G/H)$ is given by 
			\begin{align*}
			T_{H}^{\xi} \tau(xH, \xi)= {\xi}(x)^* \sum_{[\eta] \in \widehat{G/H}} d_{\eta}  \int_{G/H}  \operatorname{Tr}[(\Gamma_{\eta}(y H) \sigma_T(yH, \eta) )^*\Gamma_{\eta}(x H) ]\Gamma_{\xi}(y H)d\mu(yH)
			\end{align*}
			which is eventually same as
			$$  T_{H}^{\xi} \tau(xH, \xi)= {\xi}(x)^* \sum_{[\eta] \in \widehat{G/H}} d_{\eta}  \left({\overline{\operatorname{Tr}(\sigma_T(\cdot, \eta) ^* \Gamma_{\eta}(\cdot ) ^*\Gamma_{\eta}(x H))}^{\wedge}(\xi)}\right )^*$$
			for all $(xH, \xi) \in G / H \times \widehat{G / H}.$
		\end{thm}
		\begin{pf}
			Suppose that \(T : L^{p_{1}}(G/H) \rightarrow L^{p_{2}}(G/H)\) is $r$-nuclear operator for \(1 \leq p_{1}, p_{2}<\infty .\) Then by Theorem \ref{55} , there exist sequences \(\left\{g_{k}\right\}_{k=1}^{\infty}\) in \(L^{p_{1}^{\prime}}(G/H)\) and \(\left\{h_{k}\right\}_{k=1}^{\infty}\) in \(L^{p_{2}}(G/H)\) such that
			
			$$
			\sum_{k=1}^{\infty}\left\|g_{k}\right\|_{ L^{p_{1}^{\prime}}(G/H)}^r\left\|h_{k}\right\|_{ L^{p_{2}}(G/H)}^r<\infty
			$$
			and for all  $(yH, \eta) \in G / H \times \widehat{G / H}$ we have 
			$$\Gamma_{\eta}(y H) \sigma_T(yH, \eta) = \sum_{k=1}^{\infty} h_{k}(yH) { \widehat{\overline{g_k}}(\eta) ^*}$$
			or $$(\Gamma_{\eta}(y H) \sigma_T(yH, \eta) )^*= \sum_{k=1}^{\infty}\overline{h_{k}(yH)} { \widehat{\overline{g_k}}(\eta) }.$$
			Let $(xH, \xi) \in G / H \times \widehat{G / H}$. Then 
			\begin{align*}
			&\int_{G/H}  \operatorname{Tr}[(\Gamma_{\eta}(y H) \sigma_T(yH, \eta) )^*\Gamma_{\eta}(x H) ]\Gamma_{\xi}(y H)~d\mu(yH) \\
			&=\int_{G/H}   \operatorname{Tr}\left [ \sum_{k=1}^{\infty}\overline{h_{k}(yH)} { \widehat{\overline{g_k}}(\eta) } \Gamma_{\eta}(x H) \right ]\Gamma_{\xi}(y H)~d\mu(yH) \\
			&=\sum_{k=1}^{\infty} \operatorname{Tr}({ \widehat{\overline{g_k}}(\eta) } \Gamma_{\eta}(x H))    \int_{G/H} \overline{h_{k}(yH)} \Gamma_{\xi}(y H)~d\mu(yH) \\
			&=\sum_{k=1}^{\infty}  \widehat{{h_k}}(\xi)^* \operatorname{Tr}[{ \widehat{\overline{g_k}}(\eta) } \Gamma_{\eta}(x H)].
			\end{align*}
			Thus by  Theorem  \ref{66},
			\begin{align*}
			&\sum_{[\eta] \in \widehat{G/H}} d_{\eta}  \int_{G/H}  \operatorname{Tr}[(\Gamma_{\eta}(y H) \sigma_T(yH, \eta) )^*\Gamma_{\eta}(x H) ]\Gamma_{\xi}(y H)~d\mu(yH) \\
			&= \sum_{[\eta] \in \widehat{G/H}} d_{\eta}  \left( \sum_{k=1}^{\infty}  \widehat{{h_k}}(\xi)^* \operatorname{Tr}[{ \widehat{\overline{g_k}}(\eta) } \Gamma_{\eta}(x H)]\right)\\
			&= \sum_{k=1}^{\infty}  \widehat{{h_k}}(\xi)^* \sum_{[\eta] \in \widehat{G/H}} d_{\eta}  \operatorname{Tr}[{ \widehat{\overline{g_k}}(\eta) } \Gamma_{\eta}(x H)]\\
			&= \sum_{k=1}^{\infty}  \widehat{{h_k}}(\xi)^* { {\overline{g_k}}}(x H) =\Gamma_{\xi}(x H) \tau(xH, \xi)
			\end{align*}
			for all $(xH, \xi) \in G / H \times \widehat{G / H}$.
		\end{pf}
		
		Another criterion for the self-adjointness of $r$-nuclear pseudo-differential operators on  homogeneous space of compact groups is as follows.
		\begin{cor}
			Let $0<r \leq 1$. 	Let \(\sigma_T\) be a matrix-valued function on \(G/H \times \widehat{G/H}\) such that  \(T : L^{2}(G/H) \rightarrow L^{2}(G/H)\) is $r$-nuclear. Then \(T : L^{2}(G/H) \rightarrow L^{2}(G/H)\) is self-adjoint if and only if 
			$$  T_{H}^{\xi} \sigma_T(xH, \xi)= \xi(x)^* \sum_{[\eta] \in \widehat{G/H}} d_{\eta}  \left({\overline{\operatorname{Tr}(\sigma_T(\cdot, \eta) ^* \Gamma_{\eta}(\cdot )^*  \Gamma_{\eta}(x H))}^{\wedge}(\xi)}\right)^*$$
			for all $(xH, \xi) \in G / H \times \widehat{G / H}.$
		\end{cor}

		Now, we show that the product of a nuclear pseudo-differential operator on $L^p(G/H)$ with a bounded operator  again a nuclear pseudo-differential operator on $L^p(G/H)$  for $1 \leq  p < \infty$, where $G$ is a compact (Hausdorff) group and $H$ is a closed subgroup of $G.$ We present a formula for the symbol of product operator. The main theorem of this section is stated below. 
		
		\begin{thm}\label{99}
			Let  \(T : L^{p}(G/H) \rightarrow L^{p}(G/H)\), $1\leq p<\infty,$ be a nuclear operator with a matrix valued symbol $\sigma_T$ and let $S:  L^{p}(G/H) \rightarrow L^{p}(G/H)$ be a bounded linear operator  with symbol $\sigma_S$. Then $ST:  L^{p}(G/H) \rightarrow L^{p}(G/H)$ is a nuclear operator with symbol $\lambda$ given by
			\begin{align*}
			T_{H}^{\xi}\lambda(xH, \xi)
			&={\xi}(x)^* \sum_{k=1}^{\infty} h_{k}'(xH)    \widehat{\overline{g_k}}(\xi) ^*
			\end{align*}
			for all $(xH, \xi)\in G/H\times \widehat{G / H}$, where $\left\{g_{k}\right\}_{k=1}^{\infty}$ and $\left\{h_{k}\right\}_{k=1}^{\infty}$  are two sequences in   \( L^{p^{\prime}}(G/H)\)
			and \( L^{p}(G/H)\) respectively such that
			$
			\sum_{k=1}^{\infty}\left\|g_{k}\right\|_{ L^{p^{\prime}}(G/H)}\left\|h_{k}\right\|_{{L^{p}(G/H)}}<\infty
			$ with $$h_{k}'(xH)  =\sum_{[\eta] \in \widehat{G/H}}   d_{\eta}  \operatorname{Tr}\left [ \Gamma_{\eta}(x H) \sigma_S(xH, \eta)  \widehat{h_{k}}(\eta)  \right ]~~xH\in G/H.  $$
			
		\end{thm}
		
		\begin{pf}
			Since  \(T : L^{p}(G/H) \rightarrow L^{p}(G/H)\) is  a nuclear pseudo-differential operator for $1\leq p <\infty$, from Theorem \ref{55}, there exist  sequences  \(\left\{g_{k}\right\}_{k=1}^{\infty} \in L^{p^{\prime}}(G/H)\)
			and \(\left\{h_{k}\right\}_{k=1}^{\infty} \in L^{p}(G/H)\) such that
			$$
			\sum_{k=1}^{\infty}\left\|g_{k}\right\|_{ L^{p^{\prime}}(G/H)}\left\|h_{k}\right\|_{{L^{p}(G/H)}}<\infty
			$$ and 
			$$T_{H}^{\xi} \sigma_T(xH, \xi)= {\xi}(x)^*\sum_{k=1}^{\infty} h_{k}(xH) { \widehat{\overline{g_k}}(\xi) ^*}, ~~ (xH, \xi) \in G / H \times \widehat{G / H}.
			$$

			Let $f \in L^P(G/H)$. Then 
			\begin{align*}
			(ST f)(xH)
			&=\sum_{[\eta] \in \widehat{G/H}} d_{\eta} \operatorname{Tr}(\Gamma_{\eta}(x H) \sigma_S(xH, \eta) \widehat{Tf}(\eta))\\
			&=\sum_{[\eta] \in \widehat{G/H}} d_{\eta} \operatorname{Tr}\left [ \Gamma_{\eta}(x H)  \sigma_S(xH, \eta) \left( \int_{G/H} Tf(yH) \Gamma_{\eta}(yH)^* d\mu(yH)\right ) \right ]\\
			&=\sum_{[\eta] \in \widehat{G/H}} d_{\eta} \operatorname{Tr}\left [ \Gamma_{\eta}(x H) \sigma_S(xH, \eta)  \int_{G/H} \right.\\
			&\left. \quad\quad \quad\times \left( \sum_{[\xi] \in \widehat{G/H}} d_{\xi} \operatorname{Tr}(\Gamma_{\xi}(y H) \sigma_T(yH, \xi) \widehat{f}(\xi)) \right) \Gamma_{\eta}(yH)^*  d\mu(yH)\right ]\\
			\end{align*}
			for all $xH\in G/H$. Using the nuclearity of $T$, we have
			\begin{align*}
			(STf)(xH)
			&=\sum_{[\eta] \in \widehat{G/H}} d_{\eta} \operatorname{Tr}\left [ \Gamma_{\eta}(x H) \sigma_S(xH, \eta)  \int_{G/H}  \sum_{[\xi] \in \widehat{G/H}} d_{\xi}\right. \\
			&\quad \quad \quad  \times \left.  \operatorname{Tr}\left(\sum_{k=1}^{\infty} h_{k}(yH) { \widehat{\overline{g_k}}(\xi) ^*} \widehat{f}(\xi)\right)  \Gamma_{\eta}(yH)^*  d\mu(yH) \right ]\\
			&=\sum_{[\eta] \in \widehat{G/H}} d_{\eta} \operatorname{Tr}\left [ \Gamma_{\eta}(x H) \sigma_S(xH, \eta)   \sum_{[\xi] \in \widehat{G/H}}  \sum_{k=1}^{\infty} d_{\xi} \operatorname{Tr}\left(  { \widehat{\overline{g_k}}(\xi) ^*} \widehat{f}(\xi)\right)\right.\\
			&\quad \quad \quad  \times \left.  \int_{G/H} h_{k}(yH)  \Gamma_{\eta}(yH)^*d\mu(yH) \right]\\
			&=\sum_{[\eta] \in \widehat{G/H}} d_{\eta} \operatorname{Tr}\left [ \Gamma_{\eta}(x H) \sigma_S(xH, \eta) \sum_{[\xi] \in \widehat{G/H}}  \sum_{k=1}^{\infty} d_{\xi}  \widehat{h_{k}}(\eta) \operatorname{Tr}\left(  { \widehat{\overline{g_k}}(\xi) ^*} \widehat{f}(\xi)\right)   \right ]\\
			&=\sum_{[\eta] \in \widehat{G/H}}    \sum_{k=1}^{\infty}  \sum_{[\xi] \in \widehat{G/H}}  d_{\xi} d_{\eta} \operatorname{Tr}\left [ \Gamma_{\eta}(x H) \sigma_S(xH, \eta)   \widehat{h_{k}}(\eta)   \right ] \operatorname{Tr}\left(  { \widehat{\overline{g_k}}(\xi) ^*} \widehat{f}(\xi)\right) \\
			&=\sum_{[\xi] \in \widehat{G/H}}  d_{\xi}   \operatorname{Tr}\left(    \Gamma_{\xi}(xH)\lambda(xH, \xi)\widehat{f}(\xi)\right),
			\end{align*}	 
			where 
			\begin{align*}
			T_{H}^{\xi}\lambda(xH, \xi)&={\xi}(x)^*\sum_{k=1}^{\infty}  \sum_{[\eta] \in \widehat{G/H}}   d_{\eta}  \operatorname{Tr}\left [ \Gamma_{\eta}(x H) \sigma_S(xH, \eta)   \widehat{h_{k}}(\eta)   \right ] \widehat{\overline{g_k}}(\xi) ^*\\
			&={\xi}(x)^* \sum_{k=1}^{\infty} h_{k}'(xH)    \widehat{\overline{g_k}}(\xi) ^*
			\end{align*}
			for all $(xH, \xi)\in G/H\times \widehat{G / H}$, where  $$h_{k}'(xH)  =\sum_{[\eta] \in \widehat{G/H}}   d_{\eta}  \operatorname{Tr}\left [ \Gamma_{\eta}(x H) \sigma_S(xH, \eta)  \widehat{h_{k}}(\eta)  \right],  ~~xH \in G/H.$$
		\end{pf}
		
		\section{Applications to heat kernels} 
		
		In this section, we assume that $G$ is a compact Lie group and $H$ is a closed subgroup of $G.$  Let $\mathcal{L}_G$ be the Laplace-Beltrami operator (or the Casimir element of the universal enveloping algebra) on $G.$ For every $[\xi] \in \widehat{G},$ the matrix elements of $\xi$ are the eigenfunctions of $\mathcal{L}_G$ with same eigenvalue denoted by $-\lambda_{[\xi]}^2.$ Therefore, $$-\mathcal{L}_G \xi_{ij}= - \lambda_{[\xi]}^2 \xi_{ij}\,\,\,\, \text{for all}\, 1 \leq i,\,j \leq d_{\xi}.$$ 
		Let $-\mathcal{L}_{G/H}: C^\infty(G/H) \rightarrow C^\infty(G/H)$ be the differential operator on $G/H$ obtained by  $-\mathcal{L}_G$ acting on functions that are constant on cosets of $G,$ i.e., such that $-\widetilde{\mathcal{L}_{G/H}f}= -\mathcal{L}_G\widetilde{f}$ for $f \in C^\infty(G/H),$ where for $f \in C^\infty(G/H),\,\, \tilde{f} \in C^\infty(G)$ is the lifting of $f$ given by $\tilde{f}(x)=f(xH).$ The operator $-\mathcal{L}_{G/H}$ has the eigenspace $\Gamma_{\xi_{ij}}{(xH)}$ for $1 \leq i,\,j \leq d_\xi$ corresponding to the common eigenvalue $-\lambda_{[\xi]}^2.$ For more details on $-\mathcal{L}_{G/H}$ see \cite{NRT16}. We make use the symbol of the heat kernel $e^{-t\mathcal{L}_{G/H}}.$ Indeed, by taking in to account $\sigma_{e^{-t\mathcal{L}_{G/H}}}(xH, \xi)= e^{-t |\xi|^2}T_H^{\xi},$ where $|\xi|=\lambda_{[\xi]}^2,$ we have 
		\begin{align*}
		e^{-t\mathcal{L}_{G/H}}f(xH) &= \sum_{[\xi] \in \widehat{G/H}} d_\xi \text{Tr}( \Gamma_{\xi}(xH) \sigma_{e^{-t\mathcal{L}_{G/H}}}(xH, \xi) \widehat{f}(\xi) ) \\&= \sum_{[\xi] \in \widehat{G/H}} d_\xi \text{Tr}( \Gamma_{\xi}(xH) e^{-t \lambda_{[\xi]}^2}T_H^{\xi} \widehat{f}(\xi) ) \\&= \sum_{[\xi] \in \widehat{G/H}} d_\xi e^{-t \lambda_{[\xi]}^2} \text{Tr}( \Gamma_{\xi}(xH) \widehat{f}(\xi) ).
		\end{align*}
		
		Now, we show that the nuclearity of heat kernel on $L^p$-spaces.
		\begin{thm}
			Let $G$ be a compact Lie group and let $H$ be closed subgroup of $G$. Then the heat kernel  $e^{-t\mathcal{L}_{G/H}}:L^{p_1}(G/H) \rightarrow L^{p_2}(G/H)$ is nuclear for every $t>0$ and all $1 \leq p_1,\,p_2<\infty.$ Moreover, if $0<r \leq 1,$ then $e^{-t\mathcal{L}_{G/H}}:L^p(G/H) \rightarrow L^p(G/H)$ is $r$-nuclear operator for every $t>0$ and $1 \leq p <\infty.$ In particular, on each $L^p(G/H),$ we have the following nuclear trace formula 
			$$ \operatorname{Tr}\,(e^{-t\mathcal{L}_{G/H}})= \sum_{[\xi] \in \widehat{G/H}} d_{\xi} e^{-t\lambda_{[\xi]}^2}\, \operatorname{Tr}\,(T_H^{\xi}). $$
		\end{thm}
		\begin{proof}
			The kernel of $e^{-t \mathcal{L}_{G/H}}$ is given by 
			\begin{align*}
			K_t(x,y)&= \sum_{[\xi] \in \widehat{G/H}} d_{\xi} e^{-t \lambda_{[\xi]}^2} \operatorname{Tr}(\Gamma_\xi(xH) \Gamma(yH)^*) \\ &= \sum_{[\xi] \in \widehat{G/H}} d_{\xi} e^{-t \lambda_{[\xi]}^2} \operatorname{Tr}(\Gamma_\xi(xH) \xi(y)^*)
			\end{align*}
			with 
			$$ \operatorname{Tr}(\Gamma_\xi(xH) \xi(y)^*)= \sum_{i,j}^{d_\xi} \Gamma_\xi(xH)_{ij}\, \overline{\xi(y)}_{ij}.$$
			We set 
			$$h_{\xi, ij}= d_\xi e^{-t \lambda_{[\xi]}^2} \Gamma_{\xi}(xH)_{ij}\,\,\,\,\,\,\,\, g_{\xi, ij}= \overline{\xi(y)}_{ij}.$$
			We $p_1'$ denotes the Lebesgue conjugate of $p_1$ and $\tilde{q}_1= \text{max}\,\{2, p_1'\}$ then by Lemma \ref{vis2.4} we get 
			$$ \|g_{\xi, ij}\|_{L^{p_1'}(G/H)}=\|\overline{\xi}_{ij}\|_{L^{p_1'}(G/H)} \leq \|\Gamma_{\xi}(\cdot)_{ij}\|_{L^{p_1'}(G/H)} \leq d_\xi^{-\frac{1}{\tilde{q}_1}}.$$
			
			Also, we have 
			\begin{align*}
			\|h_{\xi, ij}\|_{L^{p_2}(G/H)}&= \|d_\xi e^{-t \lambda_{[\xi]}^2} \Gamma_{\xi}(xH)_{ij}\|_{L^{p_2}(G/H)} \\&\leq \|d_\xi e^{-t \lambda_{[\xi]}^2} \|\Gamma_{\xi}(\cdot)\|_{op}\|_{L^{p_2}(G/H)} \leq d_\xi e^{-t \lambda_{[\xi]}^2}.
			\end{align*}
			Therefore, 
			$$\sum_{[\xi], i, j} \|g_{\xi, ij}(\cdot)\|_{L^{p_1'}(G/H)} \|h_{\xi, ij}(\cdot)\|_{L^{p_2}(G/H)} \leq \sum_{[\xi] \in \widehat{G/H}} d_{\xi}^2 d_\xi^{-\frac{1}{\tilde{q}_1}} e^{-t \lambda_{[\xi]}^2}<\infty,$$  
			the last convergence is follows from any of the Weyl formula, see, for example \cite{AR}. Therefore, $e^{-t \mathcal{L}_{G/H}}$ is a nuclear operator. Similarly one can prove $r$-nuclearity of   $e^{-t \mathcal{L}_{G/H}}.$
			
			By Corollary \ref{500} and by using the fact that measure $\mu$ on $G/H$ is normalized, the nuclear trace formula  of $e^{-t \mathcal{L}_{G/H}}$  given by
			
			$$\operatorname{Tr}(e^{-t \mathcal{L}_{G/H}})= \int_{G/H} \sum_{[\xi] \in \widehat{G/H}} d_{\xi} \,\operatorname{Tr}(e^{-t \lambda_{[\xi]}^2} T_H^\xi)= \sum_{[\xi] \in \widehat{G/H}} d_\xi e^{-t \lambda_{[\xi]}^2}\,\operatorname{Tr}(T_H^\xi).$$
		\end{proof}

		\section{Acknowledgement}
		
		Vishvesh Kumar is supported by Odysseus I Project (by FWO, Belgium) of Prof. Michael Ruzhansky. The authors thank Prof. Michael Ruzhansky for  providing many valuable comments. Vishvesh Kumar thanks  Duv\'{a}n Cardona for many fruitful discussions. Shyam Swarup Mondal thanks the Council of Scientific and Industrial Research, India, for providing financial support.  He also thanks  his supervisor Jitendriya Swain for his support and encouragement.


\begin{thebibliography}{aaa}
			
			\normalsize
			\baselineskip=17pt
			\bibitem{Applebaum} D. Applebaum, \emph{Probability on Compact Lie Groups}, With a foreword by Herbert Heyer. Probability Theory and Stochastic Modelling, vol. 70. Springer, Cham (2014).
			
			\bibitem{BT10} E. Buzano and J. Toft, Schatten-von Neumann properties in the Weyl calculus, \emph{J. Funct. Anal.} 259(12), 3080-3114 (2010).
			
			\bibitem{Cardona} D. Cardona, On the nuclear trace of Fourier Integral Operators, \emph{ Rev. Integr. temas Mat.} 37(2), 219-249 (2019). 
			
			\bibitem{CardonaKumar1} D. Cardona and V. Kumar,  Multilinear analysis for discrete and periodic pseudo-differential operators in $L^p$ spaces,  \emph{ Rev. Integr. temas Mat.}  36(2),   151-164 (2018).
			
			\bibitem{CardonaKumar} D. Cardona and V. Kumar, $L^p$-boundedness and $L^p$-nuclearity of multilinear pseudo-differential operators on $\mathbb{Z}^n$ and the torus $\mathbb{T}^n,$ \emph{J. Fourier Anal. Appl.} (2019). https://doi.org/10.1007/s00041-019-09689-7 
			
			
			\bibitem{Carl} T. Carleman, \"Uber die Fourier koeffizienten einer stetigen Funktion, \emph{Acta Math.} 41(1), 377-384 (1916).
			
			
			\bibitem{Chen} Z. Chen and M. W. Wong, Traces of pseudo-differential operators on $S^{n-1}$, \emph{ J. Pseudo Differ. Oper. Appl.} 4(1),  13-24 (2013). 
			
			
			\bibitem{DK}	A. Dasgupta and V. Kumar, Hilbert-Schmidt and Trace Class Pseudo-differential Operators on the Abstract Heisenberg Group, (2019) submitted. https://arxiv.org/abs/1902.09869
			
			
			\bibitem{AR} A. Dasgupta and M. Ruzhansky, Gevrey functions and ultra distributions on compact Lie groups and homogeneous spaces, \emph{Bull. Sci. math.} 138(6), 756-782 (2014).
			
			\bibitem{Aparajita} A. Dasgupta and M. W. Wong, Pseudo-differential operators on the affine group, \emph{Pseudo-differential operators: groups, geometry and applications,} Trends Math., Birkhäuser/Springer, Cham, 1-14 (2017).
			
			
			
			
			\bibitem{D2} J. Delgado, The trace of nuclear operators on $L^p(\mu)$ for $\sigma$-finite Borel measures on second countable spaces, \emph{ Integral Equ. Oper. Theory} 68(1),  61-74 (2010).
			
			\bibitem{DR} J. Delgado and  M. Ruzhansky,  $L^p$-nuclearity, traces, and Grothendieck-Lidskii formula on compact Lie groups, \emph{ J. Math. Pures Appl. (9)}  102(1),  153-172 (2014).
			
			\bibitem{DRTk} J. Delgado, M. Ruzhansky and N. Tokmagambetov, Schatten classes, nuclearity and nonharmonic analysis on compact manifolds with boundary,    \emph{ J. Math. Pures Appl. (9)} 107(6), 758–783 (2017). 
			
			\bibitem{Ruz} J. Delgado and M. Ruzhansky, Schatten classes and traces on compact groups, \emph{Math. Res. Lett.} 24(4), 979-1003 (2017).
			
			\bibitem{DR18}J. Delgado and M. Ruzhansky, Fourier multipliers, symbols and nuclearity on compact manifolds,\emph{ J. Anal. Math.} 135(2), 757-800 (2018).
			
			
			\bibitem{Delgado} J. Delgado and M. W. Wong, $L^p$-nuclear pseudo-differential operators on $\mathbb{Z}$ and $\mathbb{S}^1$, \emph{Proc. Amer. Math. Soc.} 141(11), 3935-3942 (2013).
			
			\bibitem{Fara2} A. Ghani Farashahi, Abstract operator-valued Fourier transforms over homogeneous spaces of compact groups, \emph{ Groups Geom. Dyn.} 11(4),  1437-1467 (2017).
			
			\bibitem{Ghasemi} M.B. Ghaemi, M. Jamalpourbirgani and M. W. Wong, Characterization of nuclear pseudo-differential operators on $\mathbb{S}^1$ with applications to adjoints and products,\emph{ J. Pseudo-Differ. Oper. Appl.} 8,  191-201 (2017).
			
			
			\bibitem{Ghaemi} M. B. Ghaemi, M. Jamalpourbirgani and M. W. Wong, Characterizations, adjoints and products of nuclear pseudo-differential operators on compact and Hausdorff groups, \emph{Politehn. Univ. Bucharest Sci. Bull. Ser. A Appl. Math. Phys.} 79(4),  207-220 (2017).
			
			
			
			\bibitem{Groth2} A. Grothendieck, Produits Tensoriels Toplogiques et Espaces Nucl\'eaires , \emph{Mem.  Amer. Math. Soc.} 16,  (1955).  
			
			\bibitem{Groth1} A. Grothendieck, La th\'eorie de Fredholm, \emph{Bull. Soc. Math. France} 84, 319-384 (1956).
			
			
			\bibitem{Hor} L. H\"ormander, \emph{The Analysis of Linear Partial Differential Operators. $III$}, Springer-Verlag, Berlin (1985). 
			
			
			\bibitem{Kir1} A. A. Kirillov and A. D. Gvishiani, \emph{Teoremy i zadachi funktsionalnogo analiza}, 2nd edn. Nauka, Moscow (1988).
			
			\bibitem{Kis4} V. V. Kisil, Relative convolutions I. Properties and applications, \emph{Adv. Math.} 147(1), 35-73 (1999).
			
			
			\bibitem{Kis3} V.V. Kisil, \emph{Geometry of M\"obius transformations. Elliptic, parabolic and hyperbolic actions of $SL_2(\mathbb{R})$,} Imperial College Press, London (2012).
			
			\bibitem{Kis2} V. V. Kisil, Erlangen program at large: an overview, \emph{In: Advances in applied analysis}, Trends Math., Birkhäuser/Springer Basel AG, Basel,  1-94 (2012).
			
			
			\bibitem{Kis1} V. V. Kisil,  Calculus of operators: covariant transform and relative convolutions, \emph{Banach J. Math. Anal.} 8(2), 156-184 (2014) . 
			
			\bibitem{Niren} J. J. Kohn and L. Nirenberg, An algebra of pseudo-differential operators, \emph{ Comm. Pure Appl. Math.} 18, 269-305 (1965).
			
			\bibitem{Vish}	V. Kumar,  Pseudo-differential operators on homogeneous spaces of compact and Hausdorff groups, \emph{Forum Math.} 31(2), 275-282 (2019).
			
			\bibitem{KW} V. Kumar and M. W. Wong,  $C^*$-algebras, $H^*$-algebras and trace ideals of pseudo-differential operators on locally compact, Hausdorff and abelian groups, \emph{ J. Pseudo-Differ. Oper. Appl.} 10(2),  269-283 (2019).
			
			\bibitem{Lip} R. Lipsman,  Non-abelian Fourier analysis, \emph{Bull. Sci. Math. (2)} 98(4), 209-233 (1974). 
			
			\bibitem{Shahla1} S. Molahajloo and M. Pirhayati, Traces of pseudo-differential operators on compact and Hausdorff groups, \emph{J. Pseudo-Differ. Oper. Appl.} 4(3), 361-369 (2013).
			
			\bibitem{Shahla3} S. Molahajloo and K.L. Wong, Pseudo-differential operators on finite abelian groups, \emph{J. Pseudo-Differ. Oper. Appl.} 6(1), 1-9 (2015). 
			
			
			\bibitem{NRT16}	E. Nursultanov, M. Ruzhansky and  S. Tikhonov, Nikolskii inequality and Besov, Triebel-Lizorkin, Wiener and Beurling spaces on compact homogeneous manifolds,\emph{ Ann. Sc. Norm. Super. Pisa Cl. Sci. (5)} 16(3), 981-1017 (2016).
			
			\bibitem{Oloff} R. Oloff, $p$-normierte Operatorenideale,\emph{ Beitr\"age Anal.} (4), 105-108 (1972). 
			
			\bibitem{Pietsch} A. Pietsch, Grothendieck's concept of a $p$-nuclear operator, \emph{ Integral Equ. Oper. Theory} 7(2), 282-284 (1984).
			
			
			\bibitem{RuzT} M. Ruzhansky and V. Turunen, \emph{ Pseudo-differential Operators and Symmetries: Background Analysis and Advanced Topics}. Birkha\"user-Verlag, Basel (2010).
			
			\bibitem{RT13} M. Ruzhansky and V. Turunen, Global quantization of pseudo-differential operators on compact Lie groups, $SU(2),$ $3$-Sphere, and homogeneous spaces, \emph{Int. Math. Res. Not. IMRN} (11), 2439-2496 (2013).
			
			
			
			\bibitem{Shubin} M. A. Shubin, \emph{Psuedo-differential operators and spectral theory},  second edition, Springer-Verlag, Berlin (2001).
			
			
			\bibitem{Toft1} J. Toft, Schatten-von Neumann properties in the Weyl calculus and calculus of metrics on symplectic vector space, \emph{Ann. Global Anal. Geom.} 30(2), 169-209 (2006).
			
			
			\bibitem{Toft} J. Toft, Schatten properties for pseudo-differential on modulation spaces, In \emph{Pseudo-differential operators}, volume 1949 of \emph{Lecture Notes in Math.}, pages 175-202. Springer, Berlin (2008).
			
			
			\bibitem{Vilenkin} N. Ja. Vilenkin and A.U.  Klimyk, \emph{  Representation  of  Lie  Groups  and  Special  Functions.  Vol.  1.  Simplest Lie Groups, Special Functions and Integral Transforms,} Kluwer Academic Publishers Group, Dordrecht (1991).   
			
			
			
			
			\bibitem{Wong2} M.W. Wong, \emph{An Introduction to Pseudo-Differential Operators}, 3rd edn., World Scientific, Singapore (2014).
			
			
			
			
			
		\end{thebibliography}
	\end{document}